\documentclass{article}
\usepackage{amsmath,amssymb,a4wide,amsthm}
\usepackage{hyperref}

\newtheorem{thm}{Theorem}[section]

 \numberwithin{equation}{section}

 \def\Proj{\mathop{\rm Proj}\nolimits}
 \def\Spec{\mathop{\rm Spec}\nolimits}
 
 \def\deg{\mathop{\rm deg}\nolimits}

\def\GL{\mathop{\rm GL}\nolimits}

\let\phi\varphi
\let\epsilon\varepsilon

\newcommand{\BC}{{\mathbb{C}}}

\newcommand{\BF}{{\mathbb{F}}}
\newcommand{\BG}{{\mathbb{G}}}

\newcommand{\BP}{{\mathbb{P}}}
\newcommand{\BQ}{{\mathbb{Q}}}
\newcommand{\BR}{{\mathbb{R}}}

\newcommand{\BZ}{{\mathbb{Z}}}

\newbox\mybox
\def\arrover#1{\mathrel{
       \setbox\mybox=\hbox spread 1.4em
              {\hfil$\scriptstyle#1$\hfil}
       \vbox{\offinterlineskip\copy\mybox
             \hbox to\wd\mybox{\rightarrowfill}}}}

\def\larrover#1{\mathrel{
       \setbox\mybox=\hbox spread 1.4em
              {\hfil$\scriptstyle#1\vphantom{g}$\hfil}
       \vbox{\offinterlineskip\copy\mybox
             \hbox to\wd\mybox{\leftarrowfill}}}}

\def\ontoover#1{\mathrel{
       \setbox\mybox=\hbox spread 1.4em
              {\hfil$\scriptstyle#1\vphantom{g}$\hfil}
       \vbox{\offinterlineskip\copy\mybox
             \hbox to\wd\mybox{\rightarrowfill\hskip-2.8mm
                               $\rightarrow$}}}}
\def\leftontoover#1{\mathrel{
       \setbox\mybox=\hbox spread 1.4em
              {\hfil$\scriptstyle#1\vphantom{g}$\hfil}
       \vbox{\offinterlineskip\copy\mybox
             \hbox to\wd\mybox{$\leftarrow$\hskip-2.8mm
                               \leftarrowfill}}}}
\let\longto\longrightarrow

\def\Cinf{{\BC}_\infty}
\def\Kinf{K_\infty}

\begin{document}

\title{A note on Gekeler's $h$-function}
\author{Florian Breuer \footnote{Supported by grant no. IFRR96241 of the National Research Foundation of South Africa}\\
Stellenbosch University, Stellenbosch, South Africa\\
fbreuer@sun.ac.za}
\maketitle

\begin{abstract}
We give a brief introduction to Drinfeld modular forms, concentrating on the many equivalent constructions of the form $h$ of weight $q+1$ and type $1$, to which we contribute some new characterizations involving Moore determinants, and an application to the Weil pairing on Drinfeld modules. We also define Drinfeld modular functions of non-zero type and provide a moduli interpretation of these.
\end{abstract}
\maketitle


\section{Drinfeld modular forms}

We start with a brief introduction to Drinfeld modular forms, see e.g. \cite{gekeler_survey_1999} for more details.

Let $\BF_q$ denote the finite field of order $q$, and set $A=\BF_q[T]$, the ring of polynomials over $\BF_q$. 
We furthermore set $K=\BF_q(T)$, $\Kinf = \BF_q((\frac{1}{T}))$, the completion of $K$ at the infinite place, and $\Cinf = \hat{\bar{K}}_{\infty}$, the completion of an algebraic closure of $\Kinf$, which is an algebraically closed complete non-Archimedean field. 
The rings $A,K,\Kinf$ and $\Cinf$ are the function field analogues of the more usual $\BZ,\BQ,\BR$ and $\BC$. 

An {\em $A$-lattice} $\Lambda\subset\Cinf$ of rank $r\geq 1$ is an $A$-submodule of the form $\Lambda = \omega_1A + \omega_2A + \cdots + \omega_rA$, where $\omega_1,\omega_2,\dots,\omega_r\in\Cinf$ are linearly independent over $\Kinf$. To such a lattice we associate its {\em exponential function} 
\[
e_\Lambda (x) = x\prod_{0\neq\lambda\in \Lambda}\left(1-\frac{x}{\lambda}\right),
\]
and $e_\Lambda : \Cinf \to \Cinf$ is holomorphic, surjective, $\Lambda$-periodic and $\BF_q$-linear, with simple zeroes on $\Lambda$. It is the analogue of the usual exponential function when $r=1$ and of the elliptic functions when $r=2$. 
Since $e_\Lambda$ is $\BF_q$-linear, its logarithmic derivative is 
\[
\frac{1}{e_{\Lambda}(x)} = \sum_{\lambda\in\Lambda}\frac{1}{x+\lambda}.
\]

Denote by $\Cinf\{X^q\} = \{a_0X+a_1X^q + \cdots + a_nX^{q^n} \;|\; n\geq 0, \, a_0,\ldots,a_n\in\Cinf\}$ the non-commutative ring of $\BF_q$-linear polynomials over $\Cinf$, where multiplication is defined via composition of polynomials.

For each $a\in A$ the exponential function satisfies the functional equation
\[
e_{\Lambda}(ax) = \phi^\Lambda_a(e_\Lambda(x)), 
\]
where $\phi^\Lambda_a(X) \in \Cinf\{X^q\}$ has degree $q^{r\deg a}$. The map
\[
A \longto \Cinf\{X\}; \qquad a\longmapsto \phi^{\Lambda}_a(X)
\]
is an $\BF_q$-algebra monomorphism called a {\em Drinfeld module of rank $r$}, and plays the role of $\BG_{\mathrm m}$ in rank $r=1$ and of elliptic curves when $r=2$. There seems to be no classical analogue for Drinfeld modules of rank $r\geq 3$. More information on Drinfeld modules can be found in \cite[Chapter 4]{goss_basic_1996}.

The archetypal rank 1 Drinfeld module is the {\em Carlitz module}, defined~by
\[
\rho_T(X) = TX + X^q.
\]
It corresponds to the lattice $\bar\pi A$, where $\bar\pi\in\Cinf$ is defined up to a factor in $\BF_q^*$, is transcendental over $K$,  and plays the role of the classical $2\pi i \in\BC$. Any two rank 1 lattices in $\Cinf$ are homothetic, and consequently any two rank 1 Drinfeld modules are isomorphic over $\Cinf$.

In the case of rank $r=2$, every lattice is homothetic to a lattice of the form $zA+A$, where $z$ lies in the Drinfeld period domain
\[
\Omega := \Cinf - \Kinf,
\]
which is analogous to the classical upper half-plane. The group $\Gamma = \GL_2(A)$ acts on $\Omega$ via fractional linear transformations,
\[
\gamma(z) = \frac{az+b}{cz+d}, \qquad \gamma=\begin{pmatrix} a & b \\ c & d \end{pmatrix}\in\Gamma.
\]
Analogous to the classical situation, we can define: a {\em Drinfeld modular form} for $\Gamma$ of weight $k\in\BZ_+$ and type $m\in\BZ/(q-1)\BZ$ is a holomorphic function $f : \Omega \longto \Cinf$ satisfying
\begin{enumerate}
	\item[(a)] $\displaystyle f(\gamma z) = \det(\gamma)^{-m}(cz+d)^kf(z)$ for all $\gamma\in\Gamma$, and
	\item[(b)] $f(z)$ is holomorphic at infinity.
\end{enumerate}
This second condition can be understood as follows. From (a) follows that $f(z+1)=f(z)$, hence it admits a ``Fourier series''
\[
f(z) = \sum_{n\in\BZ} a_nt(z)^n, \qquad a_n\in\Cinf,
\]
where
\[
t(z) := \frac{1}{e_{\bar\pi A}(\bar\pi z)} = \frac{1}{\bar\pi e_A(z)} = \bar\pi^{-1} \sum_{a\in A}\frac{1}{z+a}
\]
is the parameter at infinity analogous to the classical $\exp(2\pi i z)$. Now (b) asserts that in the above expansion, $a_n=0$ for all $n<0$. If additionally $a_0=0$, we call $f$ a {\em cusp form}.

If we set $\gamma = \left(\begin{smallmatrix} \alpha & 0 \\ 0 & \alpha \end{smallmatrix}\right)$ with $\alpha\in\BF_q^*$, then the transformation rule (a) gives $f(z)=\alpha^{k-2m}f(z)$, so non-zero modular forms for $\Gamma$ can only exist if the type and weight satisfy $k\equiv 2m \pmod{q-1}$.

The set $M_{k,m}(\Gamma)$ of Drinfeld modular forms of weight $k$ and type $m$ forms a finite dimensional $\Cinf$-vector space.

As a first example of such forms, for $z\in\Omega$ we define the rank 2 lattice
\[
\Lambda_z = \bar{\pi}(zA+A)\subset\Cinf
\] 
and denote the associated rank 2 Drinfeld module by $\phi^z$. It is determined by
\[
\phi^z_T(X) = TX + g(z)X^q + \Delta(z)X^{q^2},
\]
whose coefficients $g(z)$ and $\Delta(z)$ are Drinfeld modular forms of type 0 and weights $q-1$ and $q^2-1$, respectively. 

The factor $\bar\pi$ in the definition of $\Lambda_z$ has been included so that the Fourier coefficients of $g$ and $\Delta$ turn out to be elements of $A$. This normalization is consistent with the second half of \cite{gekeler_coefficients_1988}; the reader is warned that normalization conventions vary across the literature. 

It turns out that $\Delta$ is a cusp form and David Goss \cite{goss_modular_1980}  has shown that the graded ring of Drinfeld modular forms of type 0 is generated by $g$ and $\Delta$:
\[
\bigoplus_{k\geq 1} M_{k,0}(\Gamma) = \Cinf[g,\Delta]
\]
and $g$ and $\Delta$ are algebraically independent over $\Cinf$.

\section{The many faces of $h$}

In this note, we are interested in Drinfeld modular forms of type 1 and weight $q+1$, the first instance of forms with non-zero type. It is known that $M_{q+1,1}(\Gamma)$ is one-dimensional, so there is only one such form, up to a constant multiple, and it is traditionally denoted $h$. It is defined as 
%
the Poincar\'e series 
\begin{equation}
h(z) := \sum_{\gamma\in H\backslash\Gamma} \det(\gamma)(cz+d)^{-q-1}t(\gamma z),
\label{eq:Poincare}
\end{equation}
where $H$ is the subgroup of $\Gamma$ of elements of the form $\left(\begin{smallmatrix} * & * \\ 0 & 1 \end{smallmatrix}\right)$.
The earliest appearance of such series in the literature appears to be in \cite[Chapter~X]{gerritzen_schottky_1980}.

In his seminal paper \cite{gekeler_coefficients_1988}, Ernst-Ulrich Gekeler proved that $h$ is a nowhere vanishing cusp form for $\Gamma$ of weight $q+1$ and type $1$, computed its first few Fourier coefficients and proved that the graded ring of modular forms of arbitrary type is generated by $g$ and $h$:
\begin{equation}
\bigoplus_{k\geq 1, \; m\in \BZ/(q-1)\BZ} M_{k,m}(\Gamma) = \Cinf[g,h],
\label{eq:gradedring}
\end{equation}
where $g$ and $h$ are algebraically independent.

Gekeler also proved a number of other characterizations of $h$, such as

%
%

\begin{equation}
h(z) = \bar\pi^{-1}\left(\frac{d}{dz} - (q-1)\frac{\Delta'(z)}{\Delta(z)}\right)g(z),
\label{eq:derivation}
\end{equation}
the {\em Serre derivative} of $g$, and
\begin{equation}
h(z)^{q-1} = -\Delta(z).
\label{eq:Deltaroot}
\end{equation}

Thus, from his product formula \cite{gekeler_product_1985} for $\Delta$,
\begin{equation}
h(z) = -t(z)\prod_{a\in A_+}f_a(t(z))^{q^2-1},
\label{eq:hproduct}
\end{equation}
where $A_+$ denotes the monic elements of $A$ and
$f_a(X) = X^{q^{\deg(a)}}\rho_a(X^{-1})$. 
See also \cite[(6.2)]{gekeler_modulare_1984}.

Furthermore,
\begin{equation}
h(z) = -\bar\pi^{-q}\left|\begin{matrix}z & 1 \\ \eta_1(z) & \eta_2(z)\end{matrix}\right|^{-1},
\label{eq:Legendre}
\end{equation}
where $\eta_1$ and $\eta_2$ are certain quasi-periodic functions \cite{gekeler_quasi-periodic_1989}. This is an analogue of the Legendre period relation, and links up with the De Rham cohomology of Drinfeld modules.

Finally, another attractive characterization of $h$ is via its {\em $A$-expansion}, due to Bartolom{\'e} L\'opez \cite{lopez_non-standard_2010}:
\begin{equation}
h(z) = -\sum_{a\in A_+}a^qt(az).
\label{eq:Aexpansion}
\end{equation}

\section{$T$-torsion and Moore determinants}

Let $V = (T^{-1}A/A)^{2} \cong \BF_q^2$. Then for each $v=(v_1T^{-1},v_2T^{-1})\in V' = V - \{0\}$ we obtain a weight one Eisenstein series for $\Gamma(T) = \ker \big(\Gamma \to \GL_2(\BF_q)\big)$,   
\begin{align*}
E_v(z) & := \bar\pi^{-1} \hspace{-15pt} \sum_{\scriptsize \begin{array}{c} a,b\in T^{-1}A \\ (a,b) \equiv v \bmod A^2 \end{array} } \frac{1}{az+b} \\
& = \frac{1}{\bar\pi e_{zA+A}(T^{-1}(v_1z+v_2))} = \frac{1}{e_{\Lambda_z}\big(\bar\pi T^{-1}(v_1z+v_2)\big)}.
\end{align*} 
Their reciprocals are the non-zero $T$-torsion points of $\phi^z$, i.e. we have
\[
\phi_T^z(X) = TX + g(z)X^q + \Delta(z)X^{q^2} = TX\prod_{v\in V'}\big(1-E_u(z)X). 
\]
In particular, we obtain
\begin{equation} 
\Delta(z) = T\prod_{v\in V'}E_v(z). \label{eq:Dprod}
\end{equation}

Gunther Cornelissen has proved \cite{cornelissen_drinfeld_1997-1} that in fact the ring of modular forms for $\Gamma(T)$ is generated by the $E_u$,
\[
\bigoplus_{k\geq 1}M_k(\Gamma(T)) = \Cinf[E_u \;|\; u\in V'];
\]
in this case there exist algebraic relations between these generators.

In \cite[(9.3)]{gekeler_coefficients_1988}, Gekeler showed that
\begin{equation}
h(z) = c\sum_{u,v \in V, \; \langle u,v \rangle = 1} E_u^q(z) E_v(z),
\label{eq:alternating}
\end{equation}
for some non-zero constant $c\in\Cinf^*$, where $\langle \cdot,\cdot \rangle$ denotes a non-degenerate alternating form on $V$. 

Our goal is to deduce some more characterizations of $h$ along these lines.

Recall that the {\em Moore determinant} (see \cite[Chapter 1.3]{goss_basic_1996}) of elements $x_1,x_2,\ldots,x_n$ of a field containing $\BF_q$ is defined by 
\[
M(x_1,x_2,\dots,x_n) = \det\left(x_i^{q^{j-1}}\right)_{1\leq i\leq n, \; 1\leq j\leq n.}
\]
Now choose an ordered basis $u,v$ for $V$, then
\begin{eqnarray*}
\phi^z_T(X) & = & T \, M\big(E_u(z)^{-1},E_v(z)^{-1},X\big)/M\big(E_u(z)^{-1},E_v(z)^{-1}\big)^q \\
            & = & \Delta(z) \, M\big(E_u(z)^{-1},E_v(z)^{-1},X\big)/M\big(E_u(z)^{-1},E_v(z)^{-1}\big),
\end{eqnarray*}
this being the unique polynomial with roots $\BF_q E_u(z)^{-1} + \BF_q E_v(z)^{-1}$, linear term $TX$ and leading coefficient $\Delta(z)$, from which we obtain 
\begin{equation}
\Delta(z) = T\, M\big(E_u(z)^{-1},E_v(z)^{-1}\big)^{1-q}
\end{equation}
and thus, by (\ref{eq:Deltaroot}),
\begin{equation}
h(z) = \sqrt[q-1]{-T} \, M\big(E_u(z)^{-1},E_v(z)^{-1}\big)^{-1}.
\end{equation}
Here $\sqrt[q-1]{-T}$ denotes a $(q-1)$th root of $-T$, which we determine next.
%

Let $u=(u_1T^{-1},u_2T^{-1}), \,v=(v_1T^{-1},v_2T^{-1}) \in V'$. Since the mapping $w\mapsto E_w(z)^{-1}$ is $\BF_q$-linear, we have
\begin{eqnarray*}
M\big(E_u(z)^{-1},E_v(z)^{-1}\big) & = & -\left|\begin{matrix} u_1 & u_2 \\ v_1 & v_2 \end{matrix}\right| \; 
M\big(E_{(0,T^{-1})}(z)^{-1},E_{(T^{-1},0)}(z)^{-1}\big) \\
& = & -\left|\begin{matrix} u_1 & u_2 \\ v_1 & v_2 \end{matrix}\right| \; 
E_{(0,T^{-1})}(z)^{-1} \prod_{\epsilon\in\BF_q} E_{(T^{-1},\epsilon T^{-1})}(z)^{-1},
\end{eqnarray*}
by the Moore Determinant Formula \cite[Cor. 1.3.7]{goss_basic_1996}.

From \cite[(2.1)]{gekeler_modulare_1984} we obtain the following expansions of $E_u(z)$ in terms of the parameter 
\[
t_T = t_{T}(z) :=  \frac{1}{e_{\bar\pi A}(T^{-1}\bar\pi z)} = \frac{1}{\bar\pi e_A(T^{-1}z)} = t(T^{-1}z)
\]
at the cusp $\infty$ of $\Gamma(T)\backslash\Omega$:
\begin{eqnarray*}
E_{(0,T^{-1})}(z) & = & \lambda_T^{-1} + o(t_T), \\
E_{(T^{-1},\epsilon T^{-1})}(z) & = & t_T(z) + o(t_T^2), \quad (\epsilon\in\BF_q),
\end{eqnarray*}
where
\[
\quad \lambda_T = \bar\pi e_A(T^{-1})\in \rho[T]
\]
is a specific $(q-1)$th root of $-T$.

Comparing this with (e.g. from (\ref{eq:hproduct}))
\[
h(z) = -t(z) + o(t^2) = -t_T(z)^q + o(t_T^{q+1}),
\]
we obtain

\begin{thm}
Let $u=(u_1T^{-1},u_2T^{-1}), \,v=(v_1T^{-1},v_2T^{-1}) \in V'$ and $\lambda_T = \bar\pi e_A(T^{-1})\in \rho[T]$. 
The following relations hold:
\begin{align}
M\big(E_u(z)^{-1},E_v(z)^{-1}\big) & = \lambda_T \left|\begin{matrix} u_1 & u_2 \\ v_1 & v_2 \end{matrix}\right| h(z)^{-1}, \quad\text{and}\label{eq:MooreDet2} \\
h(z) & = \lambda_T \, E_{(0,T^{-1})}(z) \prod_{\epsilon\in\BF_q} E_{(T^{-1},\epsilon T^{-1})}(z)  \label{eq:h1}\\
& = c\lambda_T \, \prod_{w\in\BP(V)} E_w(z), \label{eq:h2}
\end{align}
where the last product runs over a set of representatives in $V$ of the projective line $\BP(V) \cong \BP^1(\BF_q)$, and $c\in\BF_q^*$ depends on the choice of these representatives. \qed
\end{thm}

Note that we can also obtain (\ref{eq:h1}) and (\ref{eq:h2}), up to a multiplicative constant, from (\ref{eq:Deltaroot}) and (\ref{eq:Dprod}).

%
%

\section{The Weil pairing}

The {\em determinant} of a rank 2 Drinfeld module $\phi_T(X) = TX + gX^q + \Delta X^{q^2}$ is the rank 1 Drinfeld module defined by 
\[
\psi_T(X) = TX - \Delta X^q,
\]
and for each $a = a_0 + a_1T + \cdots + a_nT^n\in A$ the {\em Weil pairing}  \cite[Prop. 7.4]{van_der_heiden_weil_2004} is given by
\begin{align}
w_a : \phi[a]\times \phi[a] & \longto \psi[a], \nonumber \\ 
(x,y) & \longmapsto \sum_{i=0}^{n-1}\sum_{j=0}^{n-i-1} a_{i+j+1} M\big(\phi_{T^j}(x),\phi_{T^i}(y)\big). \label{eq:Weil}
\end{align}
In the case $a=T$, the Weil pairing is particularly simple -- it is the Moore determinant: $w_T(x,y) = M(x,y) = xy^q-x^qy$. In this light, we see that (\ref{eq:Deltaroot}) and (\ref{eq:MooreDet2}) each imply
\begin{equation}
\psi_T^z\big(\lambda_T \, h(z)^{-1}\big) = 0.
\label{eq:Determinant}
\end{equation}

The determinant of $\phi^z$ is isomorphic to the Carlitz module via
\begin{equation}
\psi^z = h(z)^{-1}\cdot\rho\cdot h(z),
\end{equation}
so its associated lattice is $L_z = \bar\pi h(z)^{-1} \, A$, and its $a$-torsion module is generated by $\bar\pi e_A(\frac{1}{a}) h(z)^{-1} $. 


When $a=T$, (\ref{eq:MooreDet2}) gives an explicit expression for the Weil pairing on~$\phi^z$. In general, we have

\begin{thm}
Let $a\in A$.
The Weil pairing on $\phi^z[a]$ is given by
\begin{equation}
w_a\Big(e_{\Lambda_z}\big(\bar\pi a^{-1}(u_1z+u_2)\big),e_{\Lambda_z}\big(\bar\pi a^{-1}(v_1z+v_2)\big)\Big)  
                  = \bar\pi e_A\big(a^{-1}(u_1v_2-u_2v_1)\big) h(z)^{-1}. \label{eq:weil}
\end{equation}
\end{thm}

\begin{proof}It suffices to show this for $(u_1,u_2)=(1,0)$ and $(v_1,v_2)=(0,1)$, since the general case follows from the alternating property of the Weil pairing.

Both sides lie in $\psi^z[a]$, hence are equal up to a multiplicative constant. This constant can be computed by comparing the first coefficients of their $t_a(z) := \bar\pi^{-1}e_A(a^{-1}z)$-expansions, as in the previous section.
\end{proof}

It is instructive to compare this with the analogous expression for the Weil pairing on an elliptic curve $E=\BC/(z\BZ+\BZ)$, see \cite[Ex. 1.15, p.89]{Silverman_advanced_1994},
\[
w_N\big(N^{-1}(u_1z+u_2),N^{-1}(v_1z+v_2)\big) = e^{2\pi i N^{-1}(u_1v_2-u_2v_1)},
\]
where the right hand side does not depend on $z$. The reason for this is that there is only one $\BG_{\mathrm m}$, whereas there are many rank 1 Drinfeld modules; the factor $h(z)^{-1}$ in (\ref{eq:weil}) serves to pick out the correct one, i.e. $\psi^z$.

\section{Modular functions of non-zero type}

A {\em modular function} of type $m$ is a meromorphic function $f : \Omega \to \Cinf$ satisfying
\begin{enumerate}
	\item[(a)] $\displaystyle f(\gamma z) = \det(\gamma)^{-m} f(z)$ for all $\gamma\in\Gamma$, and 
	\item[(b)] $f(z)$ is meromorphic at infinity, i.e. $a_n=0$ for $n\ll 0$ in the Fourier expansion of $f(z)$.
\end{enumerate}
As before, non-zero modular functions must satisfy $2m\equiv 0 \pmod{q-1}$.

It is well-known that the field of modular functions of type 0 is the rational function field over $\Cinf$ generated by
\[
j(z) = \frac{g(z)^{q+1}}{\Delta(z)}.
\]

\begin{thm}
The field of modular functions of arbitrary type is the rational function field $\Cinf(\tilde{\jmath})$ over $\Cinf$ generated by the modular function
\begin{equation}
\tilde{\jmath}(z) := \frac{g(z)^{m(q+1)/(q-1)}}{h(z)^m},
\label{eq:jtilde}
\end{equation}
of type $-m$, where $m$ is the least positive integer for which $q-1 | m(q+1)$. 

If $q$ is even, then $m=q-1$ and $\tilde{\jmath}(z)=j(z)$, whereas if $q$ is odd, we have $m=(q-1)/2$ and $\tilde{\jmath}(z)^2=j(z)$. The function $\tilde{\jmath}$ is holomorphic on $\Omega$ with a pole of order $m$ at $\infty$.
\end{thm}

\begin{proof}
Let $f$ be a modular function, with poles $z_1,\ldots,z_n\in\Omega$. Then for $l$ sufficiently large, $f(z)h(z)^l\prod_{i=1}^n\big(j(z)-j(z_i)\big)^l$ is a modular form of weight $l(q+1)$, hence by (\ref{eq:gradedring}) equals $F[g,h]$, where $F[X,Y]\in\Cinf[X,Y]$ is a homogeneous polynomial of weighted degree $l(q+1)$, where $X$ and $Y$ are assigned weights of $q-1$ and $q+1$, respectively. The result now follows easily.
\end{proof}


Suppose from now on that $q$ is odd, and consider the map
\[
\pi : \Omega \longto \BP_{\Cinf}(q-1,q+1); \quad z \longmapsto [g(z):h(z)], 
\]
where $\BP_{\Cinf}(q-1,q+1)$ is the weighted projective space \cite{beltrametti_introduction_1986} over $\Cinf$ 
with weights $(q-1, q+1)$; this is the set of $\Cinf$-valued points on $\Proj \Cinf[g,h]$.

Denote by $\Gamma_2$ the subgroup of matrices in $\Gamma$ with square determinant,
\[
\Gamma_2 = \{\gamma \in \Gamma \;|\; \det(\gamma)\in\BF_q^{*2}\}.
\] 
We easily compute that, for $z_1,z_2\in\Omega$, we have $\pi(z_1)=\pi(z_2)$ if and only if $z_1=\gamma(z_2)$ with 
$\gamma \in \Gamma_2$ 
if $g(z_1) \neq 0$ and $\gamma\in\Gamma$ if $g(z_1)=0$. 

The image of $\pi$ corresponds to the open affine where $h\neq 0$, which is 
\[
\Spec \big(\Cinf[g,h,h^{-1}]\big)_{\deg =0} = \Spec \Cinf[\tilde{\jmath}].
\] 
Thus we see that $\tilde{\jmath}$ induces a map
\[
\tilde{\jmath} : \Gamma_2\backslash\Omega \longto \Cinf
\]
which is bijective except above $0$, which has two pre-images.

Lastly, we describe a moduli interpretation of $\tilde{\jmath}$: when $\tilde{\jmath}\neq 0$ it pa\-ram\-e\-trizes isomorphism classes of Drinfeld modules decorated with $\BF_q^{*2}$-classes of $T$-torsion points on their determinant modules. 

More precisely, let $F$ be an algebraically closed extension of $K$ and fix  $0\neq \lambda_T \in \rho[T]\subset F$, a non-zero $T$-torsion point of the Carlitz module 
(equivalently, a $(q-1)$th root of $-T$). 
We consider pairs $(\phi,\lambda)$, where $\phi$ is a rank 2 Drinfeld module defined by $\phi_T(X) = TX+gX^q+\Delta X^{q^2}$ over  $F$  and $0\neq\lambda\in\psi[T]$ is a non-zero $T$-torsion point of its determinant module $\psi$. To such a pair we associate its $\tilde{\jmath}$-invariant 
\begin{equation}
\tilde{\jmath} = \left({\lambda}\,{\lambda^{-1}_T}\right)^{(q-1)/2} g^{(q+1)/2} \quad\in F.
\end{equation}
We call two such pairs, $(\phi,\lambda)$ and $(\phi',\lambda')$, isomorphic if either $g=g'=0$ or if $g\neq 0$ and there exist $c\in\Cinf^*$ and $\epsilon\in\BF_q^{*2}$ such that $\phi' = c\cdot\phi\cdot c^{-1}$ and $\lambda' = \epsilon c\lambda$. As simple computation shows that two such pairs are isomorphic over $F$ if and only if their $\tilde{\jmath}$-invariants are equal. We have shown

\begin{thm}
The $F$-valued points of $\Spec F[\tilde{\jmath}]$ parametrize isomorphism classes of pairs $(\phi,\lambda)$ described above. \qed
\end{thm}

The $j$-line $\Spec F[j]$ parametrizes isomorphism classes of Drinfeld modules, and the forgetful functor gives the double cover $\Spec F[\tilde{\jmath}]\to \Spec F[j]$, ramified above $0$, as $\tilde{\jmath}^2=j$.

Lastly, a point $[g:h]\in\BP_F(q-1,q+1)$ defines (up to isomorphism) a pair $(\phi,\lambda)$ with
%
\begin{align*}
\phi_T(X) & = TX + gX^q - h^{q-1}X^{q^2},  \\
\psi_T(X) & = TX + h^{q-1}X^{q-1}, \\
\lambda & = \lambda_T\,h^{-1} \in \psi[T],  \\
\tilde{\jmath}(\phi,\lambda) & = g^{(q+1)/2} / h^{(q-1)/2}
\end{align*}
and moreover, every pair $(\phi,\lambda)$ arises in this way.


%

%
%
%
%
%
%
%
%
%
%

\end{document}